\documentclass[reqno,11pt,centertags]{amsart}
\usepackage{amsmath,amsthm,amscd,amssymb,latexsym,upref}
\date{\today}

\newcommand{\bbD}{{\mathbb{D}}}

\newcommand{\bbR}{{\mathbb{R}}}

\newcommand{\bbC}{{\mathbb{C}}}

\newcommand{\bbT}{{\mathbb{T}}}
\DeclareMathAlphabet{\mathpzc}{OT1}{pzc}{m}{it}

\newcommand{\cE}{{\mathcal{E}}}

\newcommand{\cI}{{\mathcal{I}}}
\newcommand{\cJ}{{\mathcal{J}}}
\newcommand{\cK}{{\mathcal{K}}}

\newcommand{\cM}{{\mathcal{M}}}

\newcommand{\cS}{{\mathcal{S}}}

\newcommand{\fA}{{\mathfrak{A}}}
\newcommand{\fB}{{\mathfrak{B}}}

\newcommand{\fM}{{\mathfrak{M}}}

\newcommand{\fU}{{\mathfrak{U}}}

\newcommand{\fa}{{\mathfrak{a}}}
\newcommand{\fb}{{\mathfrak{b}}}

\newcommand{\fj}{{\mathfrak{j}}}

\newcommand{\z}{\zeta}

\newcommand{\pd}{{\partial}}

\def\restr#1{\,\vrule\,\lower1ex\hbox{$#1$}}

\def\D{\Delta}

\def\G{\Gamma}

\def\l{\lambda}

\def\L{\Lambda}

\def\s{\sigma}

\def\t{\theta}

\def\z{\zeta}

\renewcommand{\Re}{\text{\rm Re}\,}
\renewcommand{\Im}{\text{\rm Im}\,}

\newcommand{\tr}{\text{\rm tr}\,}

\allowdisplaybreaks \numberwithin{equation}{section}
\newtheorem{theorem}{Theorem}[section]

\newtheorem{lemma}[theorem]{Lemma}
\newtheorem{proposition}[theorem]{Proposition}

\theoremstyle{definition}

\newtheorem{remark}[theorem]{Remark}
\newtheorem{problem}[theorem]{Problem}

\begin{document}

\title[Szeg\H{o}'s Theorem for Canonical Systems]{Szeg\H{o}'s Theorem for Canonical Systems: the Arov Gauge and a Sum Rule}

\author[D.\ Damanik]{David Damanik}

\address{Department of Mathematics, Rice University, Houston, TX~77005, USA}

\email{damanik@rice.edu}

\thanks{D.\ D.\ was supported in part by NSF grant DMS--1700131 and by an Alexander von Humboldt Foundation research award.}

\author[B.\ Eichinger]{Benjamin Eichinger}

\address{Department of Mathematics, Rice University, Houston, TX~77005, USA}

\email{be11@rice.edu}

\thanks{B.\ E.\ was supported by the Austrian Science Fund FWF, project no: J 4138-N32.}

\author[P.\ Yuditskii]{Peter Yuditskii}

\address{Abteilung f\"ur Dynamische Systeme und Approximationstheorie, Johannes Kepler Universit\"at Linz, A-4040 Linz, Austria}

\email{Petro.Yudytskiy@jku.at}

\thanks{P.\ Yu.\ was supported by the Austrian Science Fund FWF, project no: P29363-N32.}

\date{\today}

\maketitle

\begin{abstract}
We consider canonical systems and investigate the Szeg\H{o} class, which is defined via the finiteness of the associated entropy functional. Noting that the canonical system may be studied in a variety of gauges, we choose to work in the Arov gauge, in which we prove that the entropy integral is equal to an integral involving the coefficients of the canonical system. This sum rule provides a spectral theory gem in the sense proposed by Barry Simon.
\end{abstract}

\section{Introduction}\label{s.1}

Simon's monograph \cite{S11} is centered around Szeg\H{o}'s theorem and its descendants. Since the theorem and the philosophy underlying it are important to this paper, let us briefly describe them. Formulated as \cite[Theorem~1.8.6]{S11}, Verblunsky's form of Szeg\H{o}'s theorem reads
\begin{equation}\label{e.szthm1}
\prod_{n = 0}^\infty (1 - |\alpha_n|^2) = \exp \left( \int \log (w(\theta)) \, \frac{d\theta}{2\pi} \right).
\end{equation}

Here, the setting is as follows. Choosing any probability measure $\mu$ on the unit circle $\partial \bbD = \{ z \in \bbC : |z| = 1 \}$ that is not supported on a finite set, one forms the associated monic orthogonal polynomials $\{ \Phi_n \}_{n \ge 0}$ in $L^2(\partial \bbD, d\mu)$ via the Gram-Schmidt procedure. These polynomials obey the recursive relations
$
\Phi_{n+1}(z) = z \Phi_n(z) - \bar \alpha_n \Phi_n^*(z)
$
with uniquely determined $\alpha_n \in \bbD = \{ z \in \bbC : |z| < 1 \}$, called the Verblunsky coefficients. Here, $\Phi_n^*$ arises from $\Phi_n$ by reversing the order of the coefficients of the polynomial and taking complex conjugates of them. This determines the left-hand side of \eqref{e.szthm1}. On the right-hand side of \eqref{e.szthm1}, $w$ denotes the Radon-Nikodym derivative of the absolutely continuous part of $\mu$ with respect to the normalized Lebesgue measure on the unit circle.

The identity \eqref{e.szthm1} holds for all such measures $\mu$. In particular, focusing on the property of both sides taking a non-zero value, we have the weaker statement that
\begin{equation}\label{e.szthm2}
\sum_{n = 0}^\infty |\alpha_n|^2 < \infty \; \Leftrightarrow \; \int \log (w(\theta)) \, \frac{d\theta}{2\pi} > - \infty.
\end{equation}

In the formulation \eqref{e.szthm2}, Szeg\H{o}'s theorem represents what Simon calls a \textit{gem of spectral theory} in \cite[Section~1.4]{S11}: a one-to-one correspondence between a class of coefficients and a class of measures.

Naturally, it has been a natural goal to pursue the proof of other gems of spectral theory, and more specifically, further results in the spirit of or in very close analogy to Szeg\H{o}'s theorem. This includes the Killip-Simon theorem for Hilbert-Schmidt perturbations of the free Jacobi matrix \cite{KS03} and follow-up work \cite{DKS10, KS09, Yu18} concerning Jacobi matrices with periodic or finite-gap quasi-periodic background and a continuum analog of the original result. The results obtained by 2010 are covered in Simon's monograph \cite{S11}, whereas Yuditskii's work \cite{Yu18} solved one of the major open problems motivated by the work covered in \cite{S11}.

The settings mentioned so far, orthogonal polynomials on the unit circle (closely related to CMV matrices as canonical representations of unitary operators of multiplicity one), Jacobi matrices (closely related to orthogonal polynomials on the real line), and continuum Schr\"odinger operators on the half-line (natural counterparts to Jacobi matrices with constant off-diagonal terms), include many of the popular classes of unitary or self-adjoint operators. Recently, there has been a push towards unifying the consideration of these classes under the umbrella of canonical systems; see, for example, the recent monograph \cite{R18} by Remling.

While we will make the setting explicit only in the next section, let us point out now that from the perspective of this recent push, it is a very natural goal to establish Szeg\H{o}'s theorem for canonical systems. In a pair of recent papers, \cite{BD17, BD}, Bessonov and Denisov have obtained a result in the spirit of \eqref{e.szthm2} for canonical systems. This result is a gem of spectral theory in the sense of Simon as it establishes a one-to-one correspondence between a class of coefficients and a class of measures. However, in \cite{BD17, BD} this equivalence of finiteness statements is not derived from an identity in the spirit of \eqref{e.szthm1}, but rather from a pair of inequalities.

This raises the natural question of whether there is an underlying identity in the spirit of \eqref{e.szthm1} in the general setting of canonical systems and it is the purpose of this paper to show that one indeed exists. In order to uncover it, we will have to change perspective. There is a ``gauge freedom" for canonical systems, and using some highly natural ``gauge fixing condition,'' it will not be too difficult to establish our version of  \eqref{e.szthm1}; this identity will be stated in Theorem~\ref{t.main} below.

The organization of the paper is as follows. We describe the setting and state the main result, Theorem~\ref{t.main}, in Section~\ref{s.2}. The proof of Theorem~\ref{t.main} is then given in Section~\ref{s.3}. Since the two different gauges, the one Bessonov and Denisov work in and the one in which we prove Theorem~\ref{t.main}, are obviously crucial to our discussion, we include two appendices that explain the historical origin and importance of each of them.

\subsection*{Acknowledgements}
P.~Yuditskii would like to thank David M\"uller for helpful discussions. D.~Damanik and B.~Eichinger were supported in part by Austrian Science Fund FWF, project no: P29363-N32.

\section{Setting and Main Result}\label{s.2}

Let $A(t)$ and $B(t)$ be $2 \times 2$ matrix-functions with entries from $L^1_\mathrm{loc}([0,\infty))$ 
such that
$$
A(t)\ge 0,\  B(t)^*=-B(t),
$$
and
$$
\tr A(t)\fj=\tr B(t)\fj=0, \quad\text{for}\ \
\fj=\begin{bmatrix}-1&0\\0&1
\end{bmatrix}.
$$
By a (two dimensional) \textit{canonical system} associated with these data we mean the differential equation of the form
\begin{equation}\label{20jun1}
\pd_t\fA(z,t)\fj=\fA(z,t)\left(-izA(t)+B(t)\right), \quad \fA(z,0)=I,\quad z\in\bbC.
\end{equation}

We say that $\cE(z)$ belongs to the Schur class $\cS$ if it is holomorphic in the upper half plane $\bbC_+$ and $|\cE(z)|\le 1$.

Assume that the system \eqref{20jun1} is given on the positive half-axis $\bbR_+$ and
$$
\int_0^\infty \tr A(t) \, dt = \infty.
$$
In this case we can define the \textit{Schur spectral function} $w$ as the limit
\begin{equation}\label{23may2}
w(z) = \lim_{t \to \infty} \frac{\fa_{11}(z,t)\cE(z) + \fa_{12}(z,t)}{\fa_{21}(z,t)\cE(z) + \fa_{22}(z,t)}, \quad \fA(z,t) = \begin{bmatrix} \fa_{11}(z,t) & \fa_{12}(z,t) \\ \fa_{21}(z,t) & \fa_{22}(z,t) \end{bmatrix},
\end{equation}
which does not depend on $\cE(z)\in\cS$; see \cite[Theorem XI]{dB61}.

Borrowing terminology from Yang-Mills theory, we say that $w(z)$ is an \textit{observable}, while the corresponding chain $\{\fA(z,t)\}$ possesses a \textit{gauge freedom}, that is, we can pass to an equivalent chain
$$
\fB(t,z) = \fA(z,t) \fU(t), \quad \fU(t) \in SU(1,1),
$$
with the same Schur spectral function.

Recall that a matrix $\fU$ belongs to $SU(1,1)$ if
$$
(i)\ \fU^*\fj\fU=\fj,\quad (ii)\ \det\fU(t)=1.
$$
Matrices obeying ($i$) are also called $\fj$-unitary.

One of the best known conditions fixing the gauge, $B(t)=0$, deals with the concept of $\fj$-modulus, which was the fundamental concept in Potapov's theory of $\fj$-contractive analytic matrix functions \cite{P0}, see also \cite{P1,P2}. We will call this the PdB-gauge (Potapov-de Branges). Some pertinent historical remarks are given in Appendix~\ref{a.1}.

Another normalization was especially promoted by Arov:
\begin{equation}\label{e.agauge}
\text{A-gauge:} \quad A(t) + B(t) \text{ is upper triangular and } \tr B(t)=0.
\end{equation}
This normalization arises in the theory of unitary extensions of isometries \cite{AG}. A short discussion is contained in Appendix~\ref{a.2}.

We say that the Schur spectral function $w(z)$ belongs to the \emph{Szeg\H{o} class} if the \textit{entropy functional} $\cI(w)$ \cite{AK1, AK2} is finite:
\begin{equation}\label{23may7}
\cI(w) := \frac 1 \pi \int_{\bbR} \log \frac 1{1-|w(x)|^2} \frac{dx}{1+x^2} < \infty.
\end{equation}

Recently, Bessonov and Denisov \cite{BD} found a characterization of the Szeg\H{o} class in terms of the canonical system data for the PdB-gauge. They proposed an explicit  expression $\tilde\cK(A)$ (functional) given in terms of $A(t)$ (for the exact formula see \cite{BD}) such that
$$
c_1\cI(w)\le \tilde\cK(A)\le c_2\cI(w) e^{c_2\cI(w)}
$$
with absolute positive constants $c_1, c_2$. As a consequence, $\cI(w)$ is finite if and only if $\tilde\cK(A)$ is finite.

This raises the natural question of whether the finiteness conditions on the two sides, measure and coefficients, can be expressed in such a way that the equivalence of the two finiteness statements can be traced back to an identity between the two expressions in question. The main result of this note is the following theorem, which in fact establishes such an identity.

\begin{theorem}\label{t.main}
For a canonical system \eqref{20jun1} under the A-gauge condition \eqref{e.agauge}, we have
\begin{equation}\label{25may}
\cI(w)=\int_0^\infty\left( \tr A(t)-2\sqrt{\det A(t)}\right)dt.
\end{equation}
\end{theorem}

In the proof we follow mainly the original paper \cite{AK2} of Arov and Krein. Certain technical details are taken from \cite{YU}, where the extremal entropy functional technique was applied to the character-automorphic Nehari problem.

\begin{remark}
We work with Schur spectral functions instead of Titchmarsh-Weyl (resolvent) functions. Passing to resolvent functions allows us to relate our entropy functional \eqref{23may7} to the one given in \cite{BD}. To be more precise: by the integral representation for functions with positive imaginary part, we have
\begin{align}\label{eq:FLTSchur}
\Re \frac{1 - w(z)}{1 + w(z)} = b \, \Im z + \frac{1}{\pi} \int_{\bbR} \frac{\Im z}{|x-z|^2} \, d\mu(x),
\end{align}
where $b\geq 0$ and $\int\frac{d\mu(x)}{1+x^2}<\infty$. In particular, if we assume that $w(i)>0$, this implies that
\begin{align*}
\frac{1}{1+w(i)}=\frac{1}{2}\left(1+b+\frac{1}{\pi}\int_{\bbR}\frac{d\mu(x)}{1+x^2}\right).
\end{align*}
Denoting by $\mu_{ac}'(x)$ the density of $\mu$ with respect to the Lebesgue measure, we have
\begin{align*}
\mu_{ac}'(x)=\Re \frac{1-w(x)}{1+w(x)}=\frac{1-|w(x)|^2}{|1+w(x)|^2}.
\end{align*}
Using that $1+w$ is outer, we obtain that
\begin{align*}
-\frac{1}{\pi}\int_{\bbR}\log\mu_{ac}'(x) \frac{dx}{1+x^2}=\cI(w)+2\log(1+w(i)).
\end{align*}
Hence,
\begin{align*}
\cI(w)=2\log\frac{1}{2}\left(1+b+\frac{1}{\pi}\int_{\bbR}\frac{d\mu(x)}{1+x^2}\right)-\frac{1}{\pi}\int_{\bbR}\log\mu_{ac}'(x) \frac{dx}{1+x^2}.
\end{align*}
\end{remark}

\medskip

An entire matrix function $\fA(z)$ is called $\fj$-\emph{expanding} in $\bbC_+$ if
\begin{equation}\label{23may40}
\fA(z)^*\fj\fA(z)-\fj\ge 0, \quad z\in\bbC_+.
\end{equation}
$\fA(z)$ is called $\fj$-\emph{inner} if in addition
\begin{equation}\label{23may3}
\fA(z)^* \fj \fA(z) - \fj = 0
\end{equation}
on the real axis.

\begin{remark}
A monotonic family of $\fj$-inner matrix functions $\fA(z,t)$ obeys the property
$$
\fA(z,t_2)^*\fj\fA(z,t_2)-\fj\ge \fA(z,t_1)^*\fj\fA(z,t_1)-\fj
$$
for $t_2>t_1$. As soon as $t$ is chosen in a way that $\fA(z,t)$ is differentiable, we define $\fM(z,t)=\fA(z,t)^{-1}\partial_t\fA(z,t)\fj$. Then the monotonicity property implies that $\fM(z,t)$ has positive real part, $\fM(z,t)+\fM(z,t)^*\geq 0$. Assuming in addition that $\fA$ is $\fj$-inner, entire, we obtain that
\begin{align*}
\fM(z,t)= -izA(t)+ B(t),
\end{align*}
where $ A(t)\geq 0$ and $B(t)=-B(t)^*$. Thus, we have \eqref{20jun1}.

Conversely, if $\fA(z,t)$ is given by this system, we have
\begin{align}\label{22may3}
\pd_t (\fA(t,z) \fj \fA(t,z)^* - \fj) & = \fA(z,t)(-iz A(t) + B(t) + i \overline{z} A(t) + B(t)^*) \fA(z,t)^* \nonumber \\
& = \frac{z-\overline{z}} i \fA(z,t)A(t)\fA(z,t)^* \ge 0,\quad z\in\bbC_+.
\end{align}
That is, $\{\fA(z,t)\}$ is a monotonic chain of $\fj$-inner entire matrix functions.
\end{remark}

\begin{remark}
Note that $w(z)$ is in the Schur class if
$$
\begin{bmatrix}w(z)\\ 1
\end{bmatrix}^*\fj\begin{bmatrix}w(z)\\ 1
\end{bmatrix}\ge 0,
$$
whereas $m(z)$ is a Titchmarsh-Weyl function, i.e. it has positive imaginary part in $\bbC_+$, if
$$
i\begin{bmatrix}m(z)\\ 1
\end{bmatrix}^*\cJ\begin{bmatrix}m(z)\\ 1
\end{bmatrix}\ge 0,\quad\text{for}\ \ \cJ=
\begin{bmatrix}
0& 1\\
-1& 0
\end{bmatrix}.
$$
While passing from functions in $\cS$ to functions with positive imaginary part requires a fractional linear transform, see \eqref{eq:FLTSchur}, we note that, since
$$
iU\cJ U^*=\fj,
$$
for some unitary matrix $U$, passing from $\fj$-inner to $\cJ$-inner matrix functions corresponds just to a conjugation by $U$. That is,
\begin{align*}
\fA(z)^*\fj\fA(z)-\fj\geq 0\quad\iff\quad i(\cM(z)^*\cJ\cM(z)-\cJ)\geq 0,
\end{align*}
where $\cM(z)=U^*\fA(z) U$.
\end{remark}

\section{Proof of the Main Theorem}\label{s.3}

Note that the following condition is equivalent to \eqref{23may40},
\begin{equation}\label{23may4}
\fA(z)\fj\fA(z)^*-\fj\ge 0, \quad z\in\bbC_+.
\end{equation}
Due to \eqref{23may4} we have
$$
-|\fa_{21}(z)|^2+|\fa_{22}(z)|^2\ge 1.
$$
That is,
$$
|\fa_{21}(z)/\fa_{22}(z)|^2+|1/\fa_{22}(z)|^2\le 1
$$
and both $\fa_{21}(z)/\fa_{22}(z)$ and $1/\fa_{22}(z)$ belong to $H^\infty$ in the upper half plane.

Furthermore, for the A-gauge, we have
$$
A(t) = \begin{bmatrix} a(t) & b(t) + i c(t) \\ b(t) - i c(t) & a(t) \end{bmatrix}, \quad
B(t) = \begin{bmatrix} 0 & b(t) + i c(t) \\ -b(t) + i c(t) & 0 \end{bmatrix}.
$$
Respectively, $\fA(i,t)$ is upper triangular and the main diagonal entries are positive. We set
$$
\fA(i,t)=\begin{bmatrix} \l(t)^{-1}&h(t)\\ 0&\l(t)
\end{bmatrix}, \quad \l(t)\ge 1.
$$
That such a normalization is always possible will be proved in Lemma \ref{larov}.

Due to
\begin{align*}
A(t) + B(t) & = \fA(i,t)^{-1} \pd_t \fA(i,t) \fj \\
& = \begin{bmatrix} \l(t) & -h(t) \\ 0 & \l(t)^{-1} \end{bmatrix} \begin{bmatrix} -\dot\l(t)\l(t)^{-2} & \dot h(t) \\ 0 & \dot\l(t) \end{bmatrix} \fj \\
& = \begin{bmatrix} \pd_t\log \l(t) & \l^2(t)\pd_t (h(t)/\l(t)) \\ 0 & \pd_t \log \l(t) \end{bmatrix},
\end{align*}
we have
\begin{align*}
A(t) & = \begin{bmatrix} \pd_t\log \l(t) & \frac 1 2\l^2(t)\pd_t (h(t)/\l(t)) \\ \frac 1 2\l^2(t)\pd_t (\overline{h(t)}/\l(t)) & \pd_t \log \l(t) \end{bmatrix} \\
& = \begin{bmatrix} a(t) & b(t) + i c(t) \\ b(t) - i c(t) & a(t) \end{bmatrix}
\end{align*}
and
\begin{align*}
B(t) & = \begin{bmatrix} 0 & \frac 1 2\l^2(t)\pd_t (h(t)/\l(t)) \\ -\frac 1 2\l^2(t)\pd_t (\overline{h(t)}/\l(t)) & 0 \end{bmatrix} \\
& = \begin{bmatrix} 0 & b(t) + i c(t) \\ -b(t) + i c(t) & 0 \end{bmatrix}.
\end{align*}
In particular,
\begin{equation}\label{22jun1}
\fa_{22}(i,t)=\l(t)=e^{\int_0^t a(t) dt}.
\end{equation}

According to the Potapov--de Branges theorem, for an arbitrary $w(z)$, there exists a monotonic chain $\{\fA(z,t)\}_{t\in[0,t_0)}$ of $j$-expanding matrix functions such that $\fA(z,0) = I$, and for an arbitrary $t\in (0,t_0)$, there exists a representation
\begin{equation}\label{23may1}
w(z) = \frac{\fa_{11}(z,t)\cE(z,t) + \fa_{12}(z,t)}{\fa_{21}(z,t)\cE(z,t) + \fa_{22}(z,t)},
\end{equation}
where $\cE(z,t)\in \cS$; see \cite[Lemma 9]{dB61} and \cite{dB62,dB,R18}.

Using \eqref{23may1} and \eqref{23may3} we have
$$
1 - |w(x)|^2 = \frac{1-|\cE(x,t)|^2}{|\fa_{22}(x,t)|^2 |1 + \cE(x,t)\fa_{21}(x,t)/\fa_{22}(x,t)|^2}, \quad x \in \bbR.
$$
Note that $1 + \cE(z,t)\fa_{21}(z,t)/\fa_{22}(z,t)$ is an outer function. Moreover, since $\fa_{21}(i,t)=0$ (A-gauge condition), we have
$$
\int \log \left| 1 + \cE(x,t)\fa_{21}(x,t)/\fa_{22}(x,t) \right| \frac{dx}{1+x^2} = 0.
$$
Furthermore, the function $\fa_{22}(z,t)$ is entire and  $1/\fa_{22}(z,t)$ belongs to $H^\infty$. Therefore a possible inner part
of $1/\fa_{22}(z,t)$ must be of the form $e^{i\s_t z}$, $\s_t\ge 0$. We get
$$
\frac 1 \pi \int \log |\fa_{22}(x,t)| \frac{dx}{1+x^2} = \log \fa_{22}(i,t) - \s_t.
$$
Thus,
$$
\frac 1 \pi \int_{\bbR} \log \frac 1{1-|w(x)|^2} \frac{dx}{1+x^2} = \frac 1 \pi \int_{\bbR} \log \frac 1{1-|\cE(x,t)|^2} \frac{dx}{1+x^2} + \log |\fa_{22}(i,t)|^2 - 2 \s_{t}.
$$
Assuming that \eqref{23may7} holds, and using \eqref{22jun1}, we get
\begin{equation}\label{23may 8}
\limsup_{t\to\infty} \left( \int_0^t a(t) \, dt - \s_t \right) \le \frac 1{2\pi} \int_{\bbR} \log \frac 1{1-|w(x)|^2} \frac{dx}{1+x^2}.
\end{equation}

Now we assume that
$$
\liminf_{t\to\infty} \left( \int_0^t a(t) \, dt - \s_t \right) = \lim_{t_k \to\infty} \left( \int_0^{t_k} a(t) \, dt - \s_{t_k} \right) < \infty.
$$
We define
$$
v(\z,t) = \frac{\fa_{12}(z,t)}{\fa_{22}(z,t)} \quad \text{and} \quad \phi(\z,t) = \frac{e^{-i\s_t z}}{\fa_{22}(z,t)}, \quad \z = \frac{z-i}{z+i} \in \bbD.
$$
Since on compact subsets of $\bbD$,
$$
\lim_{t \to \infty} v(\z,t) = v(\z) :=w(z)
$$
and
$$
|v(\z,t)|^2+|\phi(\z,t)|^2\le 1,
$$
we have (for $r<1$)
\begin{align*}
\int_{\bbT} \log \frac{1}{1-|v(r\z)|^2} \, dm(\z) & = \lim_{t \to \infty} \int_{\bbT} \log \frac{1}{1-|v(r\z,t)|^2} \, dm(\z) \\
& \le \lim_{t_k \to \infty} \int_{\bbT} \log \frac{1}{|\phi (r\z,t_k)|^2} \, dm(\z) \\
& = \lim_{t_k \to \infty} \log \frac{1}{|\phi (0,t_k)|^2} \\
& = 2 \lim_{t_k \to \infty} \left( \int_0^{t_k} a(t) \, dt - \s_{t_k} \right).
\end{align*}
We get
\begin{align*}
\int_{\bbT}\log\frac{1}{1-|v(\z)|^2}dm(\z) & = \int_{\bbT}\lim_{r\to 1}\log\frac{1}{1-|v(r\z)|^2} \, dm(\z) \\
& \le 2 \liminf_{t \to \infty} \left( \int_0^t a(t) \, dt - \s_t \right).
\end{align*}
In other words,
\begin{equation}\label{23may 8b}
\frac 1{2\pi}\int_{\bbR}\log\frac 1{1-|w(x)|^2}\frac{dx}{1+x^2}\le\liminf_{t\to\infty}\left(\int_0^ta(t)dt-\s_t\right).
\end{equation}
Combining the inequalities \eqref{23may 8} and \eqref{23may 8b}, we obtain the identity
\begin{equation}\label{21jun1}
2\lim_{t\to\infty}\left(\int_0^ta(t)dt-\s_t\right)=\frac 1{\pi}\int_{\bbR}\log\frac 1{1-|w(x)|^2}\frac{dx}{1+x^2}=\cI(w).
\end{equation}

Now we will compute $\s_t$. To this end we pass to the Potapov-de Branges normalization
$$
\fB(z,t) := \fA(z,t) \fA(0,t)^{-1}.
$$
Since $\fA(0,t)$ is $\fj$-unitary, we get
$$
\fB(z,t) = \fA(z,t) \fj \fA(0,t)^{*} \fj.
$$
Therefore
\begin{align*}
\pd_t \fB(z,t) \fj & = \fA(z,t) (-i z A(t) + B(t)) \fA(0,t)^* + \fA(z,t) B(t)^* \fA(0,t)^* \\
& = - iz \fB(z,t)H(t),
\end{align*}
where
$$
H(t)=\fA(0,t)A(t)\fA(0,t)^*.
$$
According to \cite[Theorem 39]{dB} we obtain
$$
\s_t=\int_0^t\sqrt{\det H(t)} dt=\int_0^t\sqrt{\det A(t)} dt=\int_0^t\sqrt{a(t)^2-b(t)^2-c(t)^2} dt.
$$
Since $2 a(t)=\tr A(t)$, together with \eqref{21jun1}
we obtain the final result \eqref{25may}.

\begin{appendix}

\section{The Potapov--de Branges Gauge and the $j$-Modulus of Potapov}\label{a.1}

In 1955 V.\ P.\ Potapov presented his theory of the multiplicative structure of $J$-contractive matrix functions \cite{P0}. Recall that an analytic matrix-function $W(\z)$ is called $\fj$-contractive if it satisfies an inequality of the form
$$
\fj - W(\z)^* \fj W(\z) \ge 0
$$
at every point $\z$ of its domain $\z\in \bbD$. We will restrict our discussion to the case of $2\times 2$ matrices, whereas Potapov's original considerations take place in arbitrary dimension.

Clearly the class of such matrices is multiplicative, and therefore any natural representation result for this class requires one to find a decomposition of a given matrix function into a product of ``elementary factors.''

A natural objective in this context is the generalization of the Riesz--Herglotz representation of functions of the Schur class,
\begin{equation}\label{23jun1}
w(\z) = e^{i \t_0} \z^n B(\z) e^{\int_{\partial \bbD} \frac{\z+t}{\z-t} \, d\sigma(t)},
\end{equation}
where $\sigma$ is a non-negative measure on the unit circle $\partial \bbD$ and $B(\z)$ is the Blaschke product
$$
B(\z) = \prod_{k \ge 1} B_k(\z), \quad B_k(\z) = \frac{|\z_k|}{\z_k} \frac{\z_k-\z}{1-\z\overline{\z_k}},\ \z_k \in \bbD, \z_k \not= 0.
$$

Any generalization of this kind requires one to overcome the following two serious obstacles:
\begin{itemize}
\item to take into consideration the non-commutativity of matrix multiplication,
\item to find a criterion for the convergence of the product of elementary factors.
\end{itemize}
Note that, unlike in the class of unitary matrices, one can easily find a $\fj$-unitary matrix of arbitrarily large norm; consider, for example,
$$
\fU=\begin{bmatrix}\cosh \phi&\sinh\phi\\
\sinh\phi&\cosh\phi
\end{bmatrix}.
$$
Moreover, it is a highly non-trivial task to control the product of $\fj$-contractive factors.

To overcome the first problem, Potapov used the concept of a multiplicative Stieltjes integrals. In connection with the second problem, he introduced the concept of $\fj$-modulus.

In the classical representation \eqref{23jun1}, in order to control, say, the convergence of the Blaschke product, one normalizes each factor to be positive at a fixed point, $B_k(0)>0$. After that the criterion for convergence is the Blaschke condition
$$
\prod_{k\ge 1} B_k(0) = \prod_{k\ge 1} |\z_k| > 0.
$$

Potapov demonstrates that any non-singular $\fj$-contractive matrix $W$ has a polar representation
\begin{equation}\label{23jun20}
W = \fU R, \quad R = e^{-H\fj},
\end{equation}
where $\fU$ is $\fj$-unitary and $H\ge 0$.

Note that $R$ is $\fj$-hermitian, that is, $\fj R=R^*\fj$, and
$$
W^*\fj W=R^*\fj R=\fj R^2.
$$
Thus, $R$ should be defined as a suitable root of $\fj W^*\fj W$. In Proposition~\ref{prO} below we provide an explicit formula due to Orlov (cf.~\cite{OR}) for the $\fj$-modulus $R$.

Thus, Potopov normalized each factor $R_k$ by the $\fj$-modulus, and after that he arrived at the rather complicated problem of controlling the product
\begin{equation}\label{23jun2}
\prod_{k\ge 1} R_k = \prod_{k \ge 1} e^{-H_k \fj}.
\end{equation}
From \cite[Theorem~9]{P0} we can extract the following theorem.

\begin{theorem}\label{thPm}
Let $H_k \ge 0$ and $\tr H_k \fj = 0$. Then the product \eqref{23jun2} converges if and only if
$$
\sum_{k \ge 1} H_k < \infty.
$$
\end{theorem}

A particular case of his main theorem on the multiplicative representation is the following statement (where we switched back from the unit disk to the upper half plane).

\begin{theorem}\cite[p.133]{P0} An entire matrix function $\fA(z)$, $\fA(0)=I$, which is $\fj$-expanding in the upper half-plane and $\fj$-unitary on the real axis, can be presented in the form
\begin{equation}\label{23jun3}
\fA(z) = \int\displaylimits_0^{\overset{\ell}{\curvearrowright}} e^{-iz H(t) \fj \, dt},
\end{equation}
where $H(t)$ is a summable non-negative definite matrix function.
\end{theorem}

\begin{remark}
In \eqref{23jun3} Potapov uses the concept of multiplicative integral, which is given by
$$
\int\displaylimits_0^{\overset{\ell}{\curvearrowright}}  e^{-iz H(t) \fj \, dt} = \lim_{\max \Delta_k \to 0} \prod_{k=1}^n e^{-iz H(\tau_k) \fj \Delta_k}
$$
for a partition
$$
0 = t_0 \le \tau_1 \le t_1 \le \dots t_{n-1} \le \tau_n \le t_n = \ell, \quad \Delta_k = t_k-t_{k-1}.
$$
Note that for $z=i$, we have a product of $\fj$-moduli. It is a well-known fact, see \cite[p. 135]{P0}, that this multiplicative integral solves the integral equation
$$
\pd_t \fA(z,t) \fj = \fA(z,t)(-iz H(t)), \quad \fA(z,t) = \int\displaylimits_0^{\overset{t}{\curvearrowright}} e^{-iz H(t) \fj \, dt}.
$$
Thus, his integral representation \eqref{23jun3} provides the existence of a solution of the inverse \textit{monodromy problem}: as soon as $\fA(z)$ is given and represents a $\fj$-inner entire matrix function, one can find a suitable $H(t)$, $0 \le t \le \ell$, so that $\fA(z) = \fA(z,\ell)$.
\end{remark}

The uniqueness problem was open until 1961 \cite{dB61}. Moreover, de Branges gave a solution of an inverse \textit{spectral problem}. As soon as
$$
\int_0^\infty \tr H(t) \, dt = \infty,
$$
the multiplicative integral
$$
\int\displaylimits_0^{\overset{\infty}{\curvearrowright}} e^{-iz H(t) \fj \, dt}
$$
diverges; nevertheless the limit of the ratio in \eqref{23may2} makes sense, and this is enough to reconstruct the canonical system under the PdB-gauge condition uniquely up to a monotonic change of the variable $t$; see \cite[Theorems XI and XII]{dB61}.

Finally we note that Potapov himself revised his Theorem~\ref{thPm}, see \cite{P1, P2}, and Orlov found an explicit representation for the $\fj$-modulus \cite{OR}.

For the reader's convenience, we give a proof of Orlov's theorem.

\begin{proposition}\label{prO}
Let $\fA$ be invertible and
$$
\G = \fj - \fA^* \fj \fA \ge 0.
$$
Then,
\begin{equation}\label{31may0}
I - \G^{1/2} \fj \G^{1/2} \ge 0
\end{equation}
and
\begin{equation}\label{31may1}
R = I - \fj \G^{1/2} \left( I + (I - \G^{1/2} \fj \G^{1/2})^{1/2} \right)^{-1} \G^{1/2}.
\end{equation}
\end{proposition}

\begin{proof}
It is known that
$$
\fj - \fA \fj \fA^* \ge 0
$$
and therefore
\begin{align*}
0 & \le \fA^* \fj ( \fj - \fA \fj \fA^* ) \fj \fA \\
& = \fA^* \fj \fA - \fA^* \fj \fA \fj \fA^* \fj \fA \\
& = \fj - \G - ( \fj - \G ) \fj ( \fj - \G ) \\
& = \G - \G \fj \G \\
& = \G^{1/2} (I - \G^{1/2} \fj \G^{1/2}) \G^{1/2}.
\end{align*}
Thus, the first statement of the proposition, \eqref{31may0}, is proved.

Now we assume that $\G > 0$. Then,
$$
R^2 = I - \fj \G = \G^{-1/2} ( I - \G^{1/2} \fj \G^{1/2}) \G^{1/2}.
$$
That is, $R^2$ is similar to the positive matrix $I - \G^{1/2}\fj\G^{1/2}\ge 0$. We get the following representation,
\begin{equation}\label{32may1}
R = \G^{-1/2} ( I - \G^{1/2} \fj \G^{1/2})^{1/2} \G^{1/2}.
\end{equation}

To pass to the general case, we use the identity
$$
X^{1/2} - I = (X - I)(X^{1/2} + I)^{-1},
$$
which holds for an arbitrary $X \ge 0$. Therefore \eqref{32may1} can be rewritten as
\begin{align*}
R & = I + \G^{-1/2} \left( (I - \G^{1/2} \fj \G^{1/2})^{1/2} - I \right) \G^{1/2} \\
& = I + \G^{-1/2} \left( (I - \G^{1/2} \fj \G^{1/2}) - I \right) \left( (I - \G^{1/2} \fj \G^{1/2})^{1/2} + I \right)^{-1} \G^{1/2} \\
& = I - \fj \G^{1/2} \left( (I - \G^{1/2} \fj \G^{1/2})^{1/2} + I \right)^{-1} \G^{1/2}.
\end{align*}
The last representation makes sense for an arbitrary $\G \ge 0$. Thus, \eqref{31may1} is proved.
\end{proof}

\section{The Arov Gauge and Unitary Extensions of an Isometry}\label{a.2}

During the final years of his life, Potapov was looking for a bridge connecting his multiplicative theory with classical interpolation problems of Nevanlinna-Pick type \cite{KP}, see also \cite{KO}.

The term classical interpolation problem follows the book \cite{AKH}, where different methods of solving these problems are presented. Among them is the reduction of interpolation problems to unitary extensions of isometric operators. M. G. Krein was one of the founders and protagonists of this approach, see especially \cite{AAK2}.

Trying to develop Potapov's approach, Kheifets and Yuditskii \cite{KY} were able to select data of the Abstract Interpolation Problem (AIP), which unified all known interpolation problems. But it clearly indicated that the best way of solving of AIP deals with the description of characteristic functions of unitary extensions of an isometry. The most elegant form of the solution to the latter problem is given in the paper \cite{AG} by Arov and Grossman, and we briefly recall its setting and solution.

By a \emph{unitary node} we mean a unitary operator $U$ acting from the Hilbert space $H\oplus E_1$ to $H\oplus E_2$. $H$ is called \textit{state space} and the spaces $E_1$ and $E_2$ are called \textit{coefficient spaces}. With this set of data we can associate an open linear dynamical system
$$
U(h_n\oplus u_n)=h_{n+1}\oplus v_n,\quad h_0=0, \quad u_n\in E_1,\ v_n\in E_2.
$$
The system $\{u_n\}$ is called \emph{input} and $\{v_n\}$ is the \emph{output}. It is easy to verify that the Fourier transforms of input and output,
$$
u(\z) = \sum_{n=0}^\infty u_n \z^n, \quad v(\z) = \sum_{n=0}^\infty v_n \z^n, \quad \z \in \bbD,
$$
are related by
$$
v(\z) = w(\z) u(\z),
$$
where $w(\z)$ is the Schur class operator valued function given by
\begin{equation}\label{23jun5}
w(\z) = w(\z,U) = P_{E_2} (I - \z UP_H)^{-1} U|_{E_1}.
\end{equation}
Here, $P_H$ and $P_{E_2}$ are the orthogonal projections onto the corresponding subspaces. We call $w(\z)$ the \emph{characteristic function} of the unitary node (with respect to the given coefficient spaces).

\begin{problem}\label{prue}
Let $V: K \oplus E_1 \to K \oplus E_2$ be an isometry with the defect spaces $N_1=N_{d_V}$ and $N_2=N_{\Delta_V}$,
$$
V : d_V \to \Delta_V, \quad K \oplus E_1 = d_V \oplus N_{d_V},\ K \oplus E_2 = \D_V \oplus N_{\D_V}.
$$
The unitary operator $U : H \oplus E_1 \to H \oplus E_2$, $K \subset H$, is called a \emph{unitary extension} of $V$ if
$$
U|_{d_V}=V.
$$
Describe the characteristic functions $w(\z)$ of unitary extensions of the given isometry $V$.
\end{problem}

To solve this problem we first define the Arov-Grossman extension
\begin{equation}\label{23jun6}
A : K \oplus E_1 \oplus N_2 \to K \oplus N_1 \oplus E_2
\end{equation}
given by
$$
A|_{d_V} = V
$$
$$
A|_{N_{d_V}}=id: N_{d_V}\to N_1,\quad A|_{N_2}=id: N_{2}\to N_{\D_V}.
$$
We point out that by making this extension, we extend the coefficient spaces, but the state space $K$ remains the same. We denote the corresponding characteristic operator function by
$$
S(\z) = P_{N_1 \oplus E_2} (I - \z AP_K)^{-1} A|_{E_1 \oplus N_2}.
$$
From the definition of the operator $A$ and the operator function $S(\z)$ one can see that
$$
P_{N_1} S(0)|_{N_2} = P_{N_1} A|_{N_2} = 0.
$$
In other words, if we split $S(\z)$ in blocks
$$
S(\z) = \begin{bmatrix} s_1(\z) & s(\z) \\ s_0(\z) & s_2(\z) \end{bmatrix} : \begin{bmatrix} E_1 \\ N_2 \end{bmatrix} \to \begin{bmatrix} N_1 \\ E_2
\end{bmatrix},
$$
then $s(0) = 0$.

\begin{theorem}\cite{AG}
The set of solutions of Problem~\ref{prue} is parametrized by the Schur class operator functions $\cE(\z)$ acting from $N_1$ to $N_2$. For an arbitrary parameter $\cE(\z)$, the solution $w(\z)$ is given by the formula
\begin{equation}\label{23jun10}
w(\z) = s_0(\z) + s_2(\z) \cE(\z) (I_{N_1} - s(\z) \cE(\z))^{-1} s_1(\z).
\end{equation}
\end{theorem}

\begin{remark}
The fractional linear transform of the form \eqref{23jun10} is called \emph{Redheffer transform}. Let us point out that the inverse operator in this formula is well-defined since $s(0)=0$, and therefore $\| s(\z) \cE(\z) \| \le |\z|$.
\end{remark}

\begin{remark}
The formula \eqref{23jun10} can be rewritten as the following algebraic identity,
$$
\begin{bmatrix} s_1(\z) & s(\z) \\ s_0(\z) & s_2(\z) \end{bmatrix} \begin{bmatrix} I \\ \cE(\z) \varphi(\z) \end{bmatrix} = \begin{bmatrix} \varphi(\z) \\ w(\z) \end{bmatrix},
$$
or
$$
\begin{bmatrix} I & -s_0(\z) \\ 0 & - s_1(\z) \end{bmatrix} \begin{bmatrix} w(\z) \\ I \end{bmatrix} = \begin{bmatrix} s_2(\z) & 0 \\ s(\z) & -I \end{bmatrix} \begin{bmatrix} \cE(\z) \\ I \end{bmatrix} \varphi(\z).
$$
As soon as $s_1(\z)$ is invertible, we can pass to the identity
$$
\begin{bmatrix} w(\z) \\ I \end{bmatrix} = W(\z) \begin{bmatrix} \cE(\z) \\ I \end{bmatrix} \varphi(\z),
$$
that is,
$$
w(\z) = (W_{11}(\z) \cE(\z) + W_{12}(\z)) (W_{21}(\z) \cE(\z) + W_{22}(\z))^{-1},
$$
where
$$
W(\z) = \begin{bmatrix} W_{11}(\z) & W_{12}(\z) \\ W_{21}(\z) & W_{22}(\z) \end{bmatrix} = \begin{bmatrix} I & -s_0(\z) \\ 0 & - s_1(\z) \end{bmatrix}^{-1} \begin{bmatrix} s_2(\z) & 0 \\ s(\z) & -I \end{bmatrix}.
$$
The last formula is called the \emph{Potapov-Ginzburg transform}. It allows us to pass from contractive to $\fj$-expanding matrix (operator) functions. The normalization condition $s(0) = 0$ in this case corresponds to $W_{21}(0) = 0$, that is, $W(0)$ is upper triangular, which is essentially the A-gauge condition.
\end{remark}

For the A-gauge, we can easily prove a counterpart of \eqref{23jun20}. That is, we prove that any $\fj$-contractive matrix with $\det \fB = 1$ can be normalized so that it will satisfy the A-gauge fixing condition corresponding to \eqref{e.agauge}.

\begin{lemma}
Let $\fB$ be a $\fj$-contractive matrix with $\det \fB = 1$. It admits a unique representation
$$
\fB = \begin{bmatrix} \fb_{11} & \fb_{12} \\ \fb_{21} & \fb_{22} \end{bmatrix} = \fA\, \fU, \quad \fA = \begin{bmatrix} \l^{-1} & h \\ 0 & \l
\end{bmatrix}, \quad \l \ge 1.
$$
where $\fU \in SU(1,1)$.
\end{lemma}

\begin{proof}
First we multiply $\fB$ by a diagonal $\fj$-unitary $\fU_1$ to have the same arguments for $\fb_{21}$ and $\fb_{22}$,
$$
\fB^{(1)}=\fB\fU_1,\quad \fU_1=\begin{bmatrix} e^{i\phi_1}&0\\0&e^{-i\phi_1}
\end{bmatrix}.
$$
After that we can make the resulting matrix upper triangular using a hyperbolic rotation $\fU_2$,
$$
\fB^{(2)} = \fB^{(1)} \fU_2, \quad \fU_2 = \begin{bmatrix} \cosh \phi_2 &  \cosh \phi_2 \\ \cosh \phi_2 & \cosh \phi_2 \end{bmatrix}.
$$
Due to $\det \fB^{(2)} = 1$, we have $\fb_{11}^{(2)} \fb_{22}^{(2)} = 1$. Therefore, multiplication by one more diagonal matrix $\fU_3 \in SU(1,1)$ makes the main diagonal entries positive. As a result, we have
\begin{align}\label{eq:july1}
\fA := \fB \fU_1 \fU_2 \fU_3 = \begin{bmatrix} \l^{-1} & h \\ 0 & \l \end{bmatrix}.
\end{align}
For uniqueness note that if a $\fj$-unitary matrix is of the form \eqref{eq:july1}, then $\l=1$ and $h=0$, see \eqref{eq:july2} below.
\end{proof}

A counterpart of Potapov's Theorem~\ref{thPm} is also the comparably easy Proposition~\ref{prAm} below.

Let us first prove a lemma.

\begin{lemma}\label{larov}
Let $\fA$ be an upper triangular matrix of the form
$$
\fA = \begin{bmatrix} \l^{-1} & h \\ 0 & \l \end{bmatrix}, \quad \l \ge 1.
$$
It is $j$-expanding if and only if
\begin{equation}\label{12jun1}
|h| \le \frac 1\l-\l.
\end{equation}
\end{lemma}

\begin{proof}
Indeed,
\begin{align}\label{eq:july2}
\fA^* j \fA - j = \begin{bmatrix} -\l^{-1} & 0 \\ - \overline h & \l \end{bmatrix} \begin{bmatrix} \l^{-1} & h \\ 0 & \l \end{bmatrix} - j = \begin{bmatrix} 1 - \l^{-2} & -h \l^{-1} \\ -\overline h \l^{-1} & \l^2 - |h|^2 - 1 \end{bmatrix}.
\end{align}
Therefore,
$$
\det \left( \fA^* j \fA - j \right) = \frac{(\l^2-1)^2}{\l^2} - (\l^2-1) \frac{|h|^2}{\l^2} - \frac{|h|^2}{\l^2} = \frac{(\l^2-1)^2}{\l^2} - |h|^2.
$$
The last expression is nonnegative if and only if \eqref{12jun1} holds.
\end{proof}

\begin{proposition}\label{prAm}
Let
$$
\fB_n = \begin{bmatrix} \L_n^{-1} & H_n \\ 0 & \L_n \end{bmatrix} := \prod_{k=1}^n \fA_k, \quad \fA_k = \begin{bmatrix} \l_k^{-1} & h_k \\ 0 & \l_k
\end{bmatrix}
$$
be a product of $j$-expanding matrices. The infinite product converges if and only if the infinite product of $\l_k$ converges.
\end{proposition}

\begin{proof}
Due to Lemma~\ref{larov},
$$
\| \fB_n \| \le \L_n^{-1} + \L_n + |H_n| \le C < \infty.
$$
For $m>n$, let
$$
\fB_{n,m} = \begin{bmatrix} \L_{n,m}^{-1} & H_{n,m} \\ 0 & \L_{n,m} \end{bmatrix} := \prod_{k=n+1}^m \fA_k.
$$
Then,
\begin{align*}
\| \fB_m - \fB_n \| & \le \|\fB_n\| \| \fB_{n,m} - I \| \\
& \le C(1 - \L_{n,m}^{-1} + \L_{n,m} - 1 + \L_{n,m} - \L_{n,m}^{-1}) \\
& = 2 C (\L_{n.m} - \L_{n,m}^{-1}).
\end{align*}
Thus, the sequence $\{\fB_n\}$ is Cauchy if and only if the sequence $\{\L_n\}$ is Cauchy.
\end{proof}

%

\end{appendix}

\end{document}